\documentclass{amsart}

\usepackage{amsfonts}
\usepackage{amssymb}
\usepackage{psfig}

\newcommand\R{{\mathbf{R}}}

\newcommand\B{{\mathbf{B}}}
\newcommand\Z{{\mathbf{Z}}}
\renewcommand\H{{\mathcal{H}}}
\renewcommand\S{{\mathcal{S}}}
\newcommand\eps{\varepsilon}
\newcommand\Fav{\operatorname{Fav}}
\newcommand\Lip{{\operatorname{Lip}}}
\newcommand\N{{\mathcal{N}}}
\newcommand\Norm{\operatorname{Norm}}
\newcommand\Nb{\operatorname{Nb}}
\newcommand\diam{\operatorname{diam}}
\newcommand\dist{\operatorname{dist}}

\theoremstyle{plain}
  \newtheorem{theorem}[subsection]{Theorem}

  \newtheorem{proposition}[subsection]{Proposition}
  
  \newtheorem{lemma}[subsection]{Lemma}
  \newtheorem{corollary}[subsection]{Corollary}

\theoremstyle{remark}
  \newtheorem{remark}[subsection]{Remark}
  
  \newtheorem{example}[subsection]{Example}

\theoremstyle{definition}
  \newtheorem{definition}[subsection]{Definition}

\include{psfig}

\begin{document}

\title[Quantitative Besicovitch projection theorem]{A quantitative version of the Besicovitch projection theorem via multiscale analysis}

\author{Terence Tao}
\address{UCLA Department of Mathematics, Los Angeles, CA 90095-1596.}
\email{tao@@math.ucla.edu}

\begin{abstract}  By using a multiscale analysis, we establish quantitative versions of the Besicovitch projection theorem (almost every projection of a purely unrectifiable set in the plane of finite length has measure zero) and a standard companion result, namely that any planar set with at least two projections of measure zero is purely unrectifiable.  We illustrate these results by providing an explicit (but weak) upper bound on the average projection of the $n^{th}$ generation of a product Cantor set.
\end{abstract}

\maketitle

\section{Introduction}

\subsection{The Besicovitch projection theorem}

The purpose of this note is to establish a quantitative version of the famous projection theorem of Besicovitch.  To state this theorem we first set out some notation.

\begin{definition}[Spherical measure]  Let $E \subset \R^2$, and let $0 \leq r_- < r_+$.  The \emph{one-dimensional restricted spherical content} $\S^1_{r_-,r_+}(E)$ of $E$ is defined to be the quantity
$$ \S^1_{r_-,r_+}(E) := \inf \sum_{B \in \B} \hbox{diam}(B)$$
where the infimum ranges over all at most countable collections $\B$ of open balls $B$ of radius $r(B) \in [r_-,r_+]$ which cover $E$.  The \emph{one-dimensional spherical measure} $\S^1(E)$ is then defined as
$$ \S^1(E) := \lim_{r_+ \to 0} \S^1_{0,r_+}(E).$$
A set $E \subset \R^2$ is \emph{$\S^1$-measurable} if we have $\S^1(F) = \S^1(F \cap E) + \S^1(F \backslash E)$ for all $F \subset \R^2$.
\end{definition}

If $E$ is compact, then it is easy to see that
$$ \S^1_{0,r_+}(E) = \lim_{r_- \to 0} \S^1_{r_-,r_+}(E)$$
and hence
\begin{equation}\label{hie}
 \S^1(E) = \lim_{r_+ \to 0} \lim_{r_- \to 0} \S^1_{r_-,r_+}(E).
\end{equation}
It is also well known (see e.g. \cite[Ch. 4,5]{mattila}) that spherical measure $\S^1$ is comparable up to constants with Hausdorff measure $\H^1$, and if $E$ is $\S^1$-measurable with $\S^1(E) < \infty$, then the restriction $d\S^1|_E$ of spherical measure to $E$ is a Radon measure.  We refer the reader to \cite{mattila} for further properties of Hausdorff and spherical measure.

\begin{definition}[Favard length]
We endow the unit circle $S^1 \subset \R^2$ with normalised arclength measure $d\sigma = \frac{d\H^1|_{S^1}}{2\pi}$, and the cylinder $\R \times S^1$ with the product $dm \times d\sigma$ of Lebesgue measure and normalised arclength measure.  Given any $(c,\omega) \in \R \times S^1$, we form the dual line 
$$ l_{c,\omega} := \{ x \in \R^2: x \cdot \omega = c \},$$
thus $\R \times S^1$ can be viewed as a double cover of the affine Grassmanian $A(2,1)$.
Given any Lebesgue measurable set $A \subset \R^2 \times S^1$ (which one can view as the cosphere bundle of $\R^2$), we define the \emph{Favard length} $\Fav(A)$ of $A$ to be the quantity
$$ \Fav(A) := m \times \sigma( \{ (c,\omega) \in \R \times S^1: (l_{c,\omega} \times \{\omega\}) \cap A \neq \emptyset \} ).$$
If $E \subset \R^2$, we adopt the convention $\Fav(E) := \Fav(E \times S^1)$.
\end{definition}

\begin{remark} The Favard length is usually defined (see e.g. \cite[p. 357]{falconer}) for subsets $E$ of $\R^2$ rather than subsets $A$ of $\R^2 \times S^1$, but it will be convenient to generalise the definition of Favard length in our arguments because most of our analysis shall take place on $\R^2 \times S^1$.  For future reference, we make the simple observations that Favard length is monotone and subadditive, thus
\begin{equation}\label{fav-sub}
\max(\Fav(A), \Fav(B)) \leq \Fav(A \cup B) \leq \Fav(A) + \Fav(B)
\end{equation}
for all $A, B \subset \R^2 \times S^1$.  Also if $A_1 \supset A_2 \supset \ldots$ is a nested sequence of compact subsets of $\R^2 \times S^1$, then we have
\begin{equation}\label{fav-mono}
\Fav( \bigcap_{n=1}^\infty A_n ) = \bigcap_{n=1}^\infty \Fav(A_n).
\end{equation}
\end{remark}

\begin{definition}[Rectifiability]  If $F: \R \to \R$ is a function, we define the \emph{Lipschitz constant}
$$ \|F\|_{\Lip(\R)} := \sup_{x \neq y} \frac{|F(x)-F(y)|}{|x-y|}.$$
We define a \emph{Lipschitz graph} in $\R^2$ to be any set $\Gamma$ of the form
$$ \Gamma = \{ x \omega_1 + F(x) \omega_2: x \in \R \}$$
where $\omega_1,\omega_2 \in S^1$ are orthonormal and $F: \R \to \R$ has finite Lipschitz constant.  We say that a set $E$ is \emph{purely unrectifiable} if $\S^1(E \cap \Gamma) = 0$ for all Lipschitz graphs $\Gamma$, or equivalently if
$$ m(\{ x \in \R: x \omega_1 + F(x) \omega_2 \in E \}) = 0$$
for all orthonormal $\omega_1, \omega_2 \in S^1$ and all Lipschitz $F: \R \to \R$, where $m$ denotes Lebesgue measure on $\R$.
\end{definition}

We can now state the Besicovitch projection theorem:

\begin{theorem}[Besicovitch projection theorem]\label{besithm}\cite[Theorem 6.13]{besi}  Let $E \subset \R^2$ be an $\S^1$-measurable set such that $\S^1(E) < \infty$ and that $E$ is purely unrectifiable.  Then $\Fav(E) = 0$.
\end{theorem}

As a corollary of this and \eqref{fav-mono}, we obtain

\begin{corollary}\label{besi-cor}  Let $E \subset \R^2$ be an $\S^1$-measurable set such that $\S^1(E) < \infty$ and that $E$ is purely unrectifiable.  Suppose also that $E = \bigcap_{n=1}^\infty E_n$ for some nested compact sets $E_1 \supset E_2 \supset \ldots$ (in particular, $E$ is also compact).  Then $\lim_{n \to \infty} \Fav(E_n) = 0$.
\end{corollary}

As an instance of this corollary, let us recall the standard example of the product Cantor set. 

\begin{example}[Cantor set]\label{cantor}  Let $K := \{ \sum_{n=1}^\infty a_n 4^{-n}: a_n \in \{0,3\}\}$ be the middle-half Cantor set, then $K \times K$ has finite $\S^1$ measure and is purely unrectifiable (see Proposition \ref{twoproj} below), and thus has zero Favard length by Theorem \ref{besithm}.  If we let 
$$K_n := \{ \sum_{k=1}^\infty a_k 4^{-k}: a_k \in \{0,3\} \hbox{ for } 1 \leq k \leq n \hbox{ and } a_k \in \{0,1,2,3\} \hbox{ for } k > n \}$$
then $K \times K =\bigcap_{n=1}^\infty K_n \times K_n$, and thus by Corollary \ref{besi-cor} we have $\Fav(K_n \times K_n) \to 0$ as $n \to \infty$.  
\end{example}

\subsection{A quantitative Besicovitch projection theorem}

The standard proof of Corollary \ref{besi-cor} does not give an explicitly quantitative bound on how quickly $\Fav(E_n)$ decays to zero.  Even in the model case of the product Cantor set in Example \ref{cantor}, a non-trivial upper bound on $\Fav(K_n \times K_n)$ was only established recently in \cite{peres}, who established a bound of the form
\begin{equation}\label{perro}
\Fav(K_n \times K_n) \leq C e^{-c \log_* n}
\end{equation}
for some explicit absolute constants $C, c > 0$, where $\log_*$ is the inverse tower function
$$\log_* y := \min \{ n \geq 0: \log^{(n)} y \leq 1 \}$$
and $\log^{(n)} y$ is the $n^{th}$ iterated logarithm, thus for instance $\log^{(3)} y = \log\log\log y$.  This weak bound was strengthened more recently \cite{nazarov} to 
\begin{equation}\label{nazzo}
\Fav(K_n \times K_n)  \leq C n^{-c}
\end{equation}
for some explicit absolute constants $C, c > 0$; in the converse direction, an easy argument establishes the lower bound $\Fav(K_n \times K_n) \geq c/n$ for some $c>0$, which is expected to be sharp.

The argument in \cite{peres} also extends to several other model examples of unrectifiable sets, but to the author's knowledge no quantitative version of Theorem \ref{besithm} in its full generality has appeared in the literature.  This is the main purpose of the current paper; we will not quite be able attain the type of bounds in \eqref{perro} (and certainly not those in \eqref{nazzo}), but we will obtain some explicit bound nonetheless (see Proposition \ref{favar} below).  These results are perhaps not so terribly interesting in their own right, but the author hopes that they do illustrate a general point, namely that the qualitative (and ostensibly ineffective) arguments coming from infinitary measure theory (e.g. using the Lebesgue differentiation theorem) can often be converted into quantitative (but rather weak) bounds by use of multiscale analysis and the pigeonhole principle (see below).

To achieve these goals, we must first obtain a quantitative version of the unrectifiability hypothesis.  This will be achieved as follows.

\begin{definition}[Rectifiability constant]\label{rec}  Let $E$ be a set, and let $\eps, M > 0$.  We define the \emph{rectifiability constant} $R_E(\eps,r,M)$ of $E$ with Lipschitz constant $M$, error tolerance $\eps$, and scale $r$ to be the quantity
$$ R_E(\eps,r,M) = \sup \frac{m( \{ x \in J: x \omega_1 + (F(x)+y) \omega_2 \in E \hbox{ for some } -\eps \leq y \leq \eps \} )}{m(J)}$$
where the supremum ranges over all orthonormal $\omega_1, \omega_2 \in S^1$, all $F: \R \to \R$ with Lipschitz constant $\|F\|_{\Lip(\R)} \leq M$, and all intervals $J \subset \R$ of length at least $r$.
\end{definition}

Clearly we have the trivial bound $R(\eps,r,M) \leq 1$.  Pure unrectifiability is the assertion that one can improve upon this bound when $\eps \to 0$:

\begin{proposition}[Equivalence of qualitative and quantitative unrectifiability]\label{unrec}  Let $E$ be a compact subset of $\R^2$.  Then $E$ is purely unrectifiable if and only if $\lim_{\eps \to 0} R(\eps,r,M) = 0$ for all $r > 0$ and $M > 0$.
\end{proposition}

\begin{proof} If $E$ is not purely unrectifiable, then $E \cap \Gamma$ has positive $\S^1$-measure for some Lipschitz graph $\Gamma$ with some Lipschitz constant $M$, and one easily verifies that $R(\eps,r,M)$ is then bounded from below uniformly in $\eps$ for every $r > 0$.  Now suppose for contradiction that $E$ is purely unrectifiable, but that there exists $r > 0$ and $M > 0$ such that $R(\eps,r,M)$ does not converge to zero as $\eps \to 0$.  Thus there exists a sequence $\eps_n \to 0$ and $\delta > 0$ such that $R(\eps_n,r,M) > \delta$ for all $n$.  From Definition \ref{rec}, there thus exists orthonormal $\omega_{1,n}, \omega_{2,n} \in S^1$, functions $F_n: \R \to \R$, and intervals $J_n \subset \R$ with $m(J_n) \geq r$ such that
$$ \frac{m( \{ x \in J_n: x \omega_{1,n} + (F_n(x)+y) \omega_{2,n} \in E \hbox{ for some } -\eps_n \leq y \leq \eps_n \} )}{m(J_n)} > \delta.$$
Since $E$ is compact, we conclude that $J_n$ must be contained in a fixed bounded set; similarly $F_n$ must take values in a fixed bounded range.  By the Bolzano-Weierstrass theorem we may thus pass to a subsequence and assume that $J_n$ converges to a fixed interval $J$ of length $m(J) \geq r$, thus $m(J_n \backslash J) + m(J \backslash J_n) \to 0$.  Similarly we may assume that $\omega_{1,n} \to \omega_1$ and $\omega_{2,n} \to \omega_2$ for some orthonormal $\omega_1,\omega_2 \in S^1$.  By the Arzela-Ascoli theorem we can also assume that $F_n$ converges uniformly to some $F$, which then has Lipschitz norm of at most $M$.  From the compactness of $E$, we see that the set
\begin{equation}\label{joe}
\{ x \in J: x \omega_1 + F(x) \omega_2 \in E \}
\end{equation}
contains the limit superior of the sets
$$ \{ x \in J_n: x \omega_{1,n} + (F_n(x)+y) \omega_{2,n} \in E \hbox{ for some } -\eps_n \leq y \leq \eps_n \},$$
in the sense that any point in the interior of $J$ which lies in infinitely many of the latter, must also lie in the former.  By Fubini's theorem we thus see that the set \eqref{joe} has positive measure, contradicting pure unrectifiability.
\end{proof}	

One might now naively hope that a quantitative version of the Besicovitch projection theorem would assert that if $E$ was an $\S^1$-measurable compact set with some bounded spherical (or Hausdorff) measure and bounded diameter, and if we had some suitable control on the rectifiability constants, then we would obtain an explicit non-trivial upper bound on $\Fav(E)$, which would go to zero as the rectifiability constants went to zero uniformly in $E$.  We do not know how to achieve this, and in fact suspect that such a statement is probably false (it is somewhat analogous to asking for quantitative bounds on the a.e. convergence in Lebesgue's differentiation theorem $\lim_{r \to 0} \frac{1}{2r} \int_{x-r}^{x+r} f(y)\ dy = f(x)$ which are uniform for all bounded $f$; such uniform bounds are well known to be impossible).  The problem is that a bound $\S^1(E) < L$ on the spherical measure of a set $E$ is still partially qualitative; it asserts that the spherical content $\S^1_{0,r_+}(E)$ is eventually less than $L$, but does not give an effective bound on the scale $r_+$ at which this occurs.  If however we make these scales explicit, we can in fact recover an effective bound, which is the main result of this paper:

\begin{definition}[Asymptotic notation] We use $X \lesssim Y$ or $X = O(Y)$ to denote the estimate $X \leq CY$ for some absolute constant $C > 0$.  We use $X \sim Y$ to denote the estimates $X \lesssim Y \lesssim X$.
\end{definition}

\begin{theorem}[Quantitative Besicovitch projection theorem]\label{main}  Let $L > 0$, and let $E \subset \R^2$ be a compact subset of the unit ball $B(0,1)$ with $\S^1(E) \leq L$.  Let $N \geq 1$ be an integer, and suppose that we can find scales
\begin{equation}\label{scale-mono}
0 < r_{N,-} \leq r_{N,+} \leq \ldots \leq r_{1,-} \leq r_{1,+} \leq 1
\end{equation}
obeying the following three properties:
\begin{itemize}
\item (Uniform length bound) For all $1 \leq n \leq N$ we have
\begin{equation}\label{length-bound}
{\mathcal H}^1_{r_{n,-}, r_{n,+}}(E) \leq L.
\end{equation}
\item (Scale separation)  For all $1 \leq n < N$ we have 
\begin{equation}\label{scale-sep}
r_{n+1,+} \leq \frac{1}{2} r_{n,-}.
\end{equation}
\item (Unrectifiability)  For all $1 \leq n < N$ we have
\begin{equation}\label{re}
 R_E( r_{n+1,+}, r_{n,-}, \frac{1}{r_{n,-}} ) \leq N^{-1/100}. 
\end{equation}
\end{itemize}
Then we have
$$ \Fav(E) \lesssim N^{-1/100} L.$$
\end{theorem}

Note that by applying \eqref{length-bound} at a single scale $n$, we only obtain the trivial bound $\Fav(E) \leq L$.  The point is that we can improve upon this bound by using multiple separated scales, as long as at each scale, $E$ is quantitatively unrectifiable relative to the next finer scale.  We remark that the factors of 100 can certainly be lowered (for instance, one can easily replace these factors with $10$) but we have exaggerated these constants in order to clarify the argument (and also because these bounds are in any event extremely poor).

We observe that Theorem \ref{main} easily implies Theorem \ref{besithm}.  Indeed, for the latter theorem one can quickly reduce to the case when $E$ is compact (basically because the restriction of spherical measure to $E$ is a Radon measure, and also because Lemma \ref{aproj} below allows one to neglect sets of small spherical measure for the purposes of computing Favard length); we can then normalise $E$ to lie in the unit ball $B(0,1)$.  Given any $N$, one can use \eqref{hie} and Proposition \ref{unrec} to iteratively construct scales \eqref{scale-mono} obeying the properties in Theorem \ref{main}, with $L$ set equal to $\S^1(E) + \eps$ for some $\eps > 0$.  Setting $N \to \infty$ we obtain Theorem \ref{besithm}.

Our proof is essentially a finitised version of the one in the book \cite{mattila} by Mattila, and in particular uses essentially the same geometric ingredients; we give it in Sections \ref{norm-sec}-\ref{conclude-sec}.  The main difficulty is to translate the qualitative components of the arguments in \cite{mattila} to a quantitative version. For instance, a fundamental fact in qualitative measure theory is that countably additive measures are continuous from below; thus if $(X,\mu)$ is a measure space and $E_0 \subset E_1 \subset E_2 \subset \ldots$ are measurable then $\mu( \bigcup_{n=1}^\infty E_n ) = \lim_{n \to \infty} \mu(E_n)$.  We will rely heavily on the following simple quantitative version of this fact.

\begin{lemma}[Pigeonhole principle]\label{pigeon}  Let $(X,\mu)$ be a measure space, and let $E_0 \subset E_1 \subset \ldots \subset E_N$ be any sequence of measurable subsets of $X$ with $N \geq 2$.  If $1/N \leq \eps \leq 1/2$, then there exists $0 \leq n < m \leq N$ with $m-n \geq \eps N$ such that $\mu( E_m \backslash E_n ) \lesssim \eps \mu(E_N)$.
\end{lemma}

\begin{proof} Let $k$ be the first integer greater than or equal to $\eps N$.  Observe that any element of $E_N$ belongs to at most $k+1 \sim \eps N$ sets of the form $E_{n+k} \backslash E_n$ for $0 \leq n \leq N-k$.  The claim then follows from the pigeonhole principle.
\end{proof}

Roughly speaking, this lemma will allow us to reduce the size of certain exceptional sets by an arbitrary factor $\eps$, at the cost of reducing the number of available scales by the same factor $\eps$.  We will generally apply this lemma with $\eps$ equal to some power of $N^{-1/100}$, ensuring that there are always plenty\footnote{One could work more efficiently here by running all the pigeonhole arguments ``in parallel'' rather than ``in series'', as is for instance done in Section \ref{conctwo}, but this only leads to a modest improvement in the final exponents, and also obscures the exposition somewhat, so we have chosen this more conceptually simple approach.} of scales available.

\subsection{A quantitative two projection theorem}

In order to apply the Besicovitch projection theorem, one of course needs to verify the hypothesis of unrectifiability; similarly, in order to apply Theorem \ref{main}, one needs some non-trivial quantitative decay rate on the rectifiability constants $R_E(\eps,r,M)$ as $\eps \to 0$.  One such tool to achieve the former is the following simple and well-known result, which among other things shows that the product Cantor set is purely unrectifiable.

\begin{proposition}[Two projection theorem]\label{twoproj}  Let $E \subset \R^2$ be a compact set such that two of its projections $E_{\omega} := \{ x \cdot \omega: x \in E \}$ and $E_{\omega'} := \{ x \cdot \omega': x \in E \}$ have measure zero, where $\omega, \omega' \in S^1$ are distinct and not antipodal.  Then $E$ is purely unrectifiable.
\end{proposition}

\begin{proof}  Suppose for contradiction that $E$ was not purely rectifiable, thus (after a rotation if necessary) there exists a Lipschitz graph $\{ (x,F(x)): x \in \R \}$ such that the (compact) set $A := \{ x \in \R: (x,F(x)) \in E \}$ had positive Lebesgue measure.  By the Lebesgue differentiation theorem, almost every $x \in A$ is a point of density of $A$, thus $\lim_{r \to 0} m( A \cap [x-r,x+r] ) / 2r = 1$.  Also, by the Radamacher differentiation theorem, $F(x)$ is differentiable for almost every $x$.  Thus we can find an $x_0 \in A$ which is a point of density and where $F$ has some derivative $F'(x_0)$.  Since $\omega, \omega'$ are distinct and not antipodal, at least one of the inner products $\omega \cdot (1, F'(x_0))$ and $\omega' \cdot (1, F'(x_0))$ is non-zero; without loss of generality we can assume $\omega \cdot (1, F'(x_0))$ is not zero.  It is then not difficult to show that the set $\{ \omega \cdot (x,F(x)): x \in A, |x-x_0| \leq \eps \}$ has positive measure for all sufficiently small $\eps > 0$.  But this set is contained in $E_\omega$, contradicting the hypothesis.
\end{proof}

It is thus natural to ask for a quantitative version of the above proposition.  Inspecting the above argument, we see that a proof of such a quantitative version is likely to require some sort of quantitative Lebesgue differentiation theorem and a quantitative Radamacher differentiation theorem.   This can in fact be done relatively easily, and leads to the following quantitative version:

\begin{theorem}[Quantitative two projection theorem]\label{twoproj-quant}  Let $E \subset B(0,1)$ be a compact set, and let $\omega, \omega' \in S^1$ be such that $\angle(\omega,\omega'), \angle(\omega,-\omega') \sim 1$, where $0 \leq \angle(\omega,\omega') \leq \pi$ denotes the signed angle between two unit vectors.  Suppose also that we have a sequence of scales
$$
0 < r_N < r_{N-1} < \ldots < r_1 \leq 1
$$
with the following properties.
\begin{itemize}
\item (Scale separation) Each $r_n$ is a negative power of two; in particular for all $1 \leq n < N$ we have
\begin{equation}\label{scalesep-2}
r_{n+1} \leq \frac{1}{2} r_n.
\end{equation}
\item (Small projections) If $E_{\omega} := \{ x \cdot \omega: x \in E \}$ and $E_{\omega'} := \{ x \cdot \omega': x \in E \}$, then we have
\begin{equation}\label{smallproj}
 m( \N_{r_{n+1}}(E_{\omega}) ), m( \N_{r_{n+1}}(E_{\omega'}) ) \leq r_n
\end{equation}
for all $1 \leq n < N$, where $\N_r(A)$ denotes the open $r$-neighbourhood of a set $A$.
\end{itemize}
Then we have
$$ R_E( r_N, 1, N^{1/100} ) \lesssim N^{-1/100}.$$ 
\end{theorem}

\begin{remark} This theorem gives a non-trivial bound on the rectifiability constant $R_E(\eps,r,M)$ when $r=1$; it is a simple matter to rescale this theorem to cover more general values of $r$, but we will not do so here.  It is also a routine matter to deduce Proposition \ref{twoproj} from Theorem \ref{twoproj-quant} using Proposition \ref{unrec} and the continuity of Lebesgue measure with respect to monotone limits; we leave the details to the reader as an exercise.  As before, the exponents $100$ can be improved significantly, but we have chosen not to do so in order to clarify the structure of the argument.
\end{remark}

The proof of this theorem is a relatively straightforward finitisation of the argument used to prove Proposition \ref{twoproj}, and we give it in Sections \ref{qlt}-\ref{conctwo}.  This simple proof will also serve as a model for the more complicated argument of the same nature used to prove Theorem \ref{main}.

\subsection{Example: the product Cantor set}

To illustrate the above theorems, we return to the Cantor set $K \times K$ and its approximants $K_n \times K_n$ described in Example \ref{cantor}.  We first use Theorem \ref{twoproj-quant} to establish a rectifiability bound:

\begin{proposition}[Rectifiability bound for product Cantor set]\label{rector}  If $n \geq m > l \geq 0$ and $1 \leq M \leq c \log^{1/100}(m-l+1)$ for some sufficiently small absolute constant $c > 0$, then
$$ R_{K_n \times K_n}( 2^{-m}, 2^{-l}, M ) \lesssim \log^{-1/100}(m-l+1).$$
\end{proposition}

\begin{proof} By rescaling we may normalise $l=0$.  Set $E := K_n \times K_n$, and let $\omega = e_1$, $\omega' = e_2$ be the standard basis.  With the notation of Theorem \ref{twoproj-quant}, one easily verifies that
$$  m( \N_{r}(E_{\omega}) ), m( \N_{r}(E_{\omega'}) ) \sim r^{1/2}$$
for all $2^{-n} \leq r \leq 1$ (this is basically the assertion that $K$ is Minkowski dimension $1/2$).  We can thus create a sequence of scales
$$ 2^{-m} < r_N < r_{N-1} < \ldots < r_1 \leq 1 $$
obeying the hypotheses of Theorem \ref{twoproj-quant} provided that $N \leq c \log m$ for some sufficiently large $c$.  The claim now follows from Theorem \ref{twoproj-quant}.
\end{proof}

We can now obtain a bound somewhat weaker than \eqref{perro} (and substantially weaker than \eqref{nazzo}).

\begin{proposition}[Favard bound for product Cantor set]\label{favar}  If $n \geq 100$, then
$$ \Fav(K_n \times K_n) \lesssim (\log_* n)^{-1/100}.$$
\end{proposition}

\begin{proof}  Let $E := \partial(K_n \times K_n)$ be the boundary of $K_n \times K_n$.  Observe that $K_n \times K_n$ and $E$ have identical projections and thus have identical Favard length.  Also observe that
$$ \S^1_{r, r}(E) \lesssim 1$$
for all $2^{-n} \leq r \leq 1$.  From this and Proposition \ref{rector}, we see that we can find a sequence of scales \eqref{scale-mono} obeying the properties needed for Theorem \ref{main}, with the $r_{j,+}=r_{j,-}=2^{m_j}$ equal to negative powers of $2$ with $0 \leq m_j \leq n$, provided that we can ensure
$$ m_{j+1} - m_j \gg N$$
and
\begin{equation}\label{mojo}
 m_{j+1} \gg 2^{C m_j^{100}}
 \end{equation}
for some large constant $C$.  This is possible as long as $N \leq c \log_* n$ for some sufficiently small $c$.  The claim now follows from Theorem \ref{main}.
\end{proof}

We remark that the bound Proposition \ref{rector} can be improved by using the machinery of $\beta$-numbers, as developed by Jones \cite{jones} (see also \cite{okikiolu}), although this does not end up significantly improving Proposition \ref{favar}.  As we will not use such improvements in this paper, we shall defer these results to the Appendix.

\subsection{Acknowledgements}

We thank Yuval Peres and Peter Jones for encouragement, to Peter Jones and Raanan Schul for pointing out the use of $\beta$-numbers to bound the rectifiability constants, and John Garnett for corrections.  The author is also indebted to the anonymous referee for many useful comments.  The author is supported by a grant from the Macarthur Foundation.

\subsection{Organisation of the paper}

The paper is organised as follows.  Sections \ref{qlt}-\ref{conctwo} are devoted to the proof of the quantitative analogue of the two projection theorem (Proposition \ref{twoproj}), namely Theorem \ref{twoproj-quant}.  We have chosen to cover this theorem first, ahead of the more difficult quantitative Besicovitch projection theorem (Theorem \ref{main}), as the arguments are somewhat simpler and thus serve as an introduction to the methods used to prove Theorem \ref{main}.

The proof of Theorem \ref{twoproj-quant} is modeled on the proof of Proposition \ref{twoproj}, but of course with all qualitative statements converted into quantitative ones.  An inspection of that proof reveals three major ingredients: the Lebesgue differentiation theorem (that establishes points of density inside a dense set), the Radamacher differentiation theorem (that establishes points of differentiability for a Lipschitz function), and then a simple observation that the image of a set under a Lipschitz function will have positive measure if there exists a point of density of that set for which the Lipschitz function is differentiable with non-zero derivative.  We will give quantitative analogues of these three facts in Sections \ref{qlt}, \ref{rad}, \ref{conctwo} respectively, thus concluding the proof of Theorem \ref{twoproj-quant}.

The remaining sections \ref{norm-sec}-\ref{conclude-sec} is devoted to the proof of Theorem \ref{main}.  The arguments broadly follow the standard proof of the projection theorem, as found for instance in \cite{mattila}, and we summarise it as follows.  In Section \ref{norm-sec} we dispose of the ``normal'' directions - the directions $\omega$ and the points $x \in E$ such that $E$ behaves somewhat like a subset of a Lipschitz graph in the direction $\omega$ near $x$.  Because of the unrectifiability of $E$, we expect the contribution of these directions to the Favard length to be small, and this is indeed what is shown in Section \ref{norm-sec}.

The remaining sections are devoted to the more difficult statement that the non-normal directions (in which $E$ has large pieces that are oriented somewhat along the direction $\omega$ near $x$) also contribute a negligible amount to the Favard length.  To do this one needs to divide the space of non-normal directions into further pieces.  In Section \ref{highmult} we dispose of the ``high multiplicity'' directions - those directions $\omega$ and points $x \in E$ such that the ray from $x$ in direction $\omega$ intersects $E$ a large number of times.  It turns out that the bounds on the length of $E$ will allow us to get a good bound on the contribution of this case to the Favard length.

We are now left with the directions whose multiplicity is bounded or zero.  It turns out that a scale pigeonholing argument allows one to dispose of the former (this is done in Section \ref{posmult}), and so it remains to handle the latter case (in which $E$ is behaving like a graph in the $\omega$ direction, albeit one with unbounded Lipschitz constant).  In Section \ref{highdens} we use the Hardy-Littlewood maximal inequality to eliminate the ``high-density strips'' of such graphs - portions of the graph which have large density inside thin rectangles oriented along $\omega$.  After eliminating such regions, it turns out that each point $x$ has only a few directions $\omega$ which are still contributing to the Favard length, giving a small net contribution in all; we detail this in Section \ref{conclude-sec}.

\section{A quantitative Lebesgue differentiation theorem}\label{qlt}

We now begin the proof of Theorem \ref{twoproj-quant}.  Let $E, \omega, \omega', N, r_1,\ldots,r_N$ be as in that Theorem.  From the trivial bound $R_E( r_N, 1, N^{1/100} ) \leq 1$ we obtain the claim for $N \lesssim 1$, so we may take $N$ to be large.  For minor notational reasons it is also convenient to assume (as we may) that $N^{1/100}$ is an integer.

Using rotation invariance and the hypothesis $E \subset B(0,1)$, it suffices to show that 
$$ m( \{ x \in [-1,1]: (x,F(x)+y) \in E \hbox{ for some } -r_N \leq y \leq r_N \} ) \lesssim N^{-1/100}$$
for all $F: \R \to \R$ with Lipschitz constant $\|F\|_{\Lip} \leq N^{1/100}$.  Let us thus fix $F$, and set
$$ A := \{ x \in [-1,1]: (x,F(x)+y) \in E \hbox{ for some } -r_N \leq y \leq r_N \}.$$
Note that $A$ is clearly a compact set; our task is to show that
\begin{equation}\label{ajob}
m(A) \lesssim N^{-1/100}.
\end{equation}
From definition of $A$ we see that
\begin{equation}\label{xfx}
 \{ (x,F(x)) \cdot \omega: x \in A \} \subset \N_{r_N}(E_\omega)
\end{equation}
and similarly with $\omega$ replaced by $\omega'$.

Inspecting the proof of Proposition \ref{twoproj}, we see that the next step should be some quantitative version of the Lebesgue differentiation theorem, which would establish plenty of points of density in $A$.  It turns out that the correct way to achieve this is to work with discretised versions of $A$.

For any $1 \leq n \leq N$, we partition $[-1,1]$ into dyadic intervals of length $r_n$ (the endpoints may overlap, but these have measure zero and will play no role).  Let $A_n$ be the union of all such dyadic intervals which have a non-empty intersection with $A$, thus
$$ A \subset A_N \subset \ldots \subset A_1 \subset [-1,1].$$
Applying Lemma \ref{pigeon} (and recalling that $N$ is large), we may thus find $0.1 N \leq n_0 \leq 0.9 N$ such that
the set
$$ \Delta A := A_{n_0 - N^{-3/100} N} \backslash A_{n_0 + N^{-3/100} N}$$
has small Lebesgue measure:
\begin{equation}\label{mdelta-small}
 m( \Delta A ) \lesssim N^{-3/100}.
\end{equation}
Roughly speaking, this will mean that most points in the coarse-scale set $A_{n_0 - N^{-3/100} N}$ are still points of density for $A_{n_0 + N^{-3/100} N}$ all the way down to the much finer scale $n_0+N^{-3/100} N$.  The set $A_{n_0 + N^{-3/100} N}$ is not quite as small as $A$ itself, so these are not points of density for $A$, but this set will nevertheless serve as a good enough proxy for $A$ for our arguments (which will take place at scales significantly coarser than $r_{n_0+N^{-3/100} N}$.

We now fix this value of $n_0$, and henceforth shall be working entirely in the range of scale indices
$$ n \in [n_0 - N^{-3/100} N, n_0 + N^{-3/100} N] \subset [1,N].$$

\section{A quantitative Radamacher differentiation theorem}\label{rad}

We continue to use Proposition \ref{twoproj} as a model for our argument.  This model suggests that the next step in the argument should involve some quantitative version of the Radamacher differentiation theorem.  This we do as follows.

Given any scale index $1 \leq n \leq N$, we can consider the dyadic grid $\{ j r_n \in [-1,1]: j \in \Z \}$ of points in $[-1,1]$, equally separated by $r_n$.  We then let $F_n:[-1,1] \to \R$ be the unique continuous, piecewise linear function which agrees with $F$ on this dyadic grid (thus $F_n(j r_n) = F(j r_n)$), and is linear on each dyadic interval $[j r_n, (j+1) r_n]$ inside $[-1,1]$.  Since $F$ has a Lipschitz norm of at most $N^{1/100}$, it is easy to see that $F_n$ does also, and furthermore we have the uniform approximation
\begin{equation}\label{ffn}
 \| F - F_n \|_{L^\infty([-1,1])} \lesssim N^{1/100} r_n.
\end{equation} 
Also, the derivative $F'_n$, which is defined almost everywhere, has an $L^2$ norm of $O(N^{1/100})$, thus
$$ 0 \leq \| F'_n \|_{L^2([-1,1])}^2 \lesssim N^{2/200}.$$
Also, one can readily check that the differences $F'_{n+1} - F'_n$ are all orthogonal to each other, and to $F'_1$.  Thus by Pythagoras' theorem, $\| F'_n \|_{L^2([-1,1])}^2$ is an increasing function of $n$.  Specialising to scale indices between $n_0 - N^{-3/100} N$ and $n_0 + N^{-3/100} N$, and using the pigeonhole principle, we can thus find a scale $n_0 - 0.9 N^{-3/100} N \leq n_1 \leq n_0 + 0.9 N^{-3/100} N$ such that
$$ \| F'_{n_1 + N^{-9/100} N} \|_{L^2([-1,1])}^2 - \| F'_{n_1 - N^{-9/100} N} \|_{L^2([-1,1])}^2 \lesssim N^{-4/100}.$$
In particular we see from Pythagoras' theorem that
\begin{equation}\label{pythag}
 \|F'_{n_1 + N^{-9/100} N} - F'_{n_1 - N^{-9/100} N} \|_{L^2([-1,1])} \lesssim N^{-2/100}.
 \end{equation}
Roughly speaking, this estimate asserts that $F$ is ``mostly differentiable'' between the coarse scale index of $n_1 - N^{-9/100}$ and the fine scale index of $n_1 + N^{-9/100}$, and will serve as our quantitative substitute for the Radamacher differentiation theorem.

We now fix this value of $n_1$, and henceforth shall be working entirely in the range of scale indices
$$ n \in [n_1 - N^{-9/100} N, n_1 + N^{-9/100} N] \subset [n_0 - N^{-3/100} N, n_0 + N^{-3/100} N] \subset [1,N].$$

\section{Conclusion of the argument}\label{conctwo}

The scale index $n_1$ established in the preceding section will now be a good place to conduct our analysis.

Let us suppose for contradiction that the conclusion of Theorem \ref{twoproj-quant} is false, then
$$ m(A) \gtrsim N^{-1/100}.$$
This implies that the set 
\begin{equation}\label{mono}
 m( A_{n_0 - N^{-3/100} N} ) \gtrsim N^{-1/100}.
\end{equation}
We now localise to a relatively large interval in $A_{n_0 - N^{-3/100} N}$ in which $A$ and $F$ are both well-behaved (this is the quantitative analogue of selecting a point $x_0$ in the proof of Proposition \ref{twoproj}).  From \eqref{mono}, \eqref{mdelta-small} we have
$$ \int_{A_{n_0 - N^{-3/100} N}} N^{2/100} 1_{\Delta A} \lesssim m( A_{n_0 - N^{-3/100} N} ).$$
Also, from \eqref{mono}, \eqref{pythag}
$$ \int_{A_{n_0 - N^{-3/100} N}} N^{3/100} |F'_{n_1 + N^{-9/100} N} - F'_{n_1 - N^{-9/100} N}|^2 \lesssim m( A_{n_0 - N^{-3/100} N} ).$$
We concatenate these two bounds together:
$$ \int_{A_{n_0 - N^{-3/100} N}} N^{2/100} 1_{\Delta A} + N^{3/100} |F'_{n_1 + N^{-9/100} N} - F'_{n_1 - N^{-9/100} N}|^2 \lesssim m( A_{n_0 - N^{-3/100} N} ).$$
The set $A_{n_0 - N^{-3/100} N}$ is the union of dyadic intervals of length $r_{n_0 - N^{-3/100} N}$, and hence also the union of dyadic intervals of length $r_{n_1 - N^{-9/100} N}$.  Thus by the pigeonhole principle, we can find a dyadic $I \subset A_{n_0 - N^{-3/100} N}$ of length $r_{n_1 - N^{-9/100} N}$, such that
$$ \int_{I} N^{2/100} 1_{\Delta A} + N^{3/100} |F'_{n_1 + N^{-9/100} N} - F'_{n_1 - N^{-9/100} N}|^2 \lesssim m(I).$$
We fix such an interval $I$.  From the above estimate we have
\begin{equation}\label{mad}
 m( \Delta A \cap I ) \lesssim N^{-2/100} m(I)
\end{equation}
and
\begin{equation}\label{flate}
 \int_I |F'_{n_1 + N^{-9/100} N} - F'_{n_1 - N^{-9/100} N}|^2 \lesssim N^{-3/100} m(I).
\end{equation}

Note that the function $F'_{n_1 - N^{-9/100} N}$ is constant on $I$, so let us write $c := F'_{n_1 - N^{-9/100} N}$ (this is the analogue of $F'(x_0)$ in the proof of Proposition \ref{twoproj}); note that $|c| \leq N^{1/100}$.  From the hypotheses on $\omega,\omega'$, we see that either
$$
|(1, c) \cdot \omega| \gtrsim 1
$$
or
$$ |(1, c) \cdot \omega'| \gtrsim 1.$$
Without loss of generality we may assume that the former holds.  By reversing the sign of $\omega$ if necessary, we shall assume that
\begin{equation}\label{c}
t := (1, c) \cdot \omega \gtrsim 1.
\end{equation}
In view of \eqref{flate}, \eqref{c}, we expect the function $x \mapsto (x,F(x)) \cdot \omega$ to be largely increasing on $I$ at scales $r_{n_1 + N^{-9/100} N}$ and above.  The plan is to combine this with \eqref{xfx} to get some good lower bounds on the size of some neighbourhood of $E_\omega$.

Let us subdivide $I$ into $S := r_{n_1 + N^{-9/100} N} / r_{n_1 - N^{-9/100} N}$ dyadic intervals $J_1,\ldots,J_S$ of length $r_{n_1 + N^{-9/100} N}$.  
We also divide $\R$ into disjoint intervals $(K_m)_{m \in \Z}$ of length $100 N^{1/100} r_{n_1 + N^{-9/100} N}$.  We then say that a dyadic interval $J_j$ of the former kind \emph{reaches} an interval $K_m$ of the latter kind if we have $(x,F_{n_1+N^{-9/100}N}(x)) \cdot \omega \in K_m$ for at least one $x \in J_j$.  Since $F_{n_1+N^{-9/100}N}$ has a Lipschitz constant of at most $N^{1/100}$, we see that each $J_j$ reaches either one or two intervals $K_m$.

The function $F'_{n_1+N^{-9/100}N}$ is constant on each interval $J_j$.  Let us say that an interval $J_j$ is \emph{good} if $|F'_{n_1+N^{-9/100}N} - c| < t/100$, and \emph{bad} otherwise, where $t$ is the quantity in \eqref{c}.  From \eqref{flate} and Chebyshev's inequality we see that at most $O( N^{-3/100} S )$ of the intervals $J_1,\ldots,J_S$ are bad.  

The key lemma is

\begin{figure}[tb]
\centerline{\psfig{figure=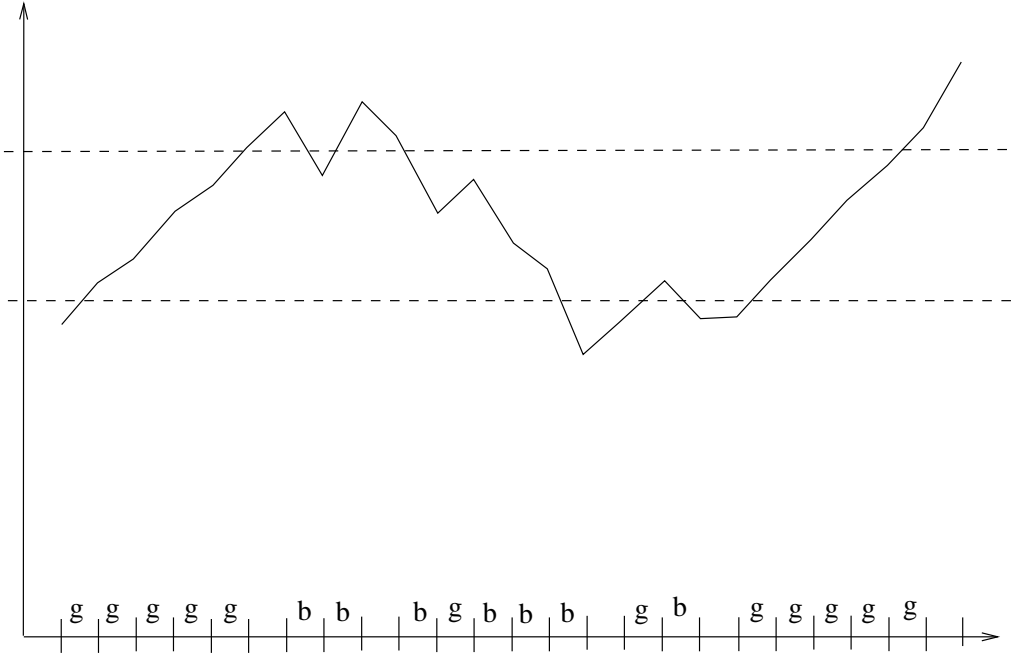}}
\caption{A graph of the function $x \mapsto (x,F_{n_1+N^{-9/100}N}(x)) \cdot \omega$, and the various intervals $J_j$ which reach a certain interval $K_m$.  The good intervals $J_j$ that reach $K_m$ are marked with ``g''; the bad ones are marked with ``b''.  The first and last strings are exceptional strings; the other two are not.}
\label{fig1}
\end{figure}

\begin{lemma}[Rising sun type lemma]\label{risi}  Let $m \in \Z$.  Then at least one of the following statements is true:
\begin{itemize}
\item[(i)] There are at most $O(N^{1/100})$ intervals $J_j$ which reach $K_m$.
\item[(ii)] Of all the intervals $J_j$ which reach $K_m$, the proportion of those $J_j$ which are bad is $\gtrsim N^{-1/100}$.
\end{itemize}
(See Figure \ref{fig1}.)
\end{lemma}

\begin{proof}  Observe that on any good interval $J_j$, the function $x \mapsto (x,F_{n_1+N^{-9/100}N}(x)) \cdot \omega$ is a linear function whose slope is positive and comparable to $1$, by \eqref{c}.  On a bad interval, this function is still linear, but the slope can have either sign and has magnitude $O( N^{1/100} )$.  On the other hand, an interval $K_m$ only has length $100 N^{1/100}$ times the length of any of the $J_j$.  As a consequence, we see that any consecutive string $J_j, J_{j+1},\ldots,J_{j+s}$ of good intervals which reach $K_m$ can have length at most $s = O(N^{1/100})$.

Now consider a maximal string $J_j, J_{j+1},\ldots,J_{j+s}$ of consecutive intervals which reach $K_m$. If this string has length much larger than $N^{1/100}$, the above discussion shows that at least $\gtrsim N^{-1/100}$ of the intervals in this string must be bad.  If the string has length $O( N^{1/100} )$ and there is at least one bad interval, then clearly $\gtrsim N^{-1/100}$ of the intervals in this string are bad.  The only remaining case is if the string has length $O( N^{1/100} )$ and consists entirely of good intervals; let us call these the ``exceptional strings''.  Then we see that any exceptional string, the function
$x \mapsto (x,F_{n_1+N^{-9/100}N}(x)) \cdot \omega$ ascends monotonically from below $K_m$ to above $K_m$.  Thus, by the intermediate value theorem, between any two exceptional strings there must be at least one bad interval that touches $K_m$; thus the number of exceptional strings cannot exceed one plus the number of bad intervals that touch $K_m$.  On the other hand, on all the non-exceptional strings we have seen that $\gtrsim N^{-1/100}$ of the intervals are bad.  Since all exceptional strings have length $O(N^{-1/100})$, the claim follows.  
\end{proof}

Call an interval $K_m$ \emph{low multiplicity} if case (i) of Lemma \ref{risi} holds (with a suitable choice of implied constant), and \emph{high multiplicity} otherwise.  We call an interval $J_j$ \emph{typical} if it only reaches low multiplicity intervals $K_m$, and \emph{atypical} if it reaches at least one high multiplicity interval $K_m$.  Observe that of all the atypical intervals $J_j$ that reach any given high multiplicity interval $K_m$, at least $\gtrsim N^{-1/100}$ of them will be bad.  Since every interval $J_j$ reaches either one or two intervals $K_m$, we thus see that $\gtrsim N^{-1/100}$ of the atypical intervals $J_j$ are bad.  Since the number of bad intervals is $O( N^{-3/100} S )$, we conclude that $O(N^{-2/100} S)$ of the intervals $J_1,\ldots,J_S$ are atypical.

Let us say that an interval $J_j$ is \emph{$A$-empty} if $A \cap J_j = \emptyset$, and \emph{$A$-nonempty} otherwise.  Note that if $J_j$ is $A$-empty, then it is completely contained in $\Delta A$, thus from \eqref{mad} we see that at most $O( N^{-2/100} S )$ of the intervals $J_1,\ldots,J_S$ are $A$-empty.  Combining this with the previous analysis, we see that all but $O( N^{-2/100} S)$ of the intervals $J_1,\ldots,J_S$ are both typical and $A$-nonempty; in partiuclar the number of typical $A$-nonempty intervals is $\sim S$.  For any such typical $A$-nonempty interval $J_j$, we thus have an element $x_j \in A \cap J_j$, and the number $(x_j,F_{n_1 + N^{-9/100}N}(x_j)) \cdot \omega$ will lie in a low-multiplicity interval $K_m$.  Since the $K_m$ are low-multiplicity, a simple counting argument then shows that the number of such intervals $K_m$ obtained in this manner must be at least $\gtrsim N^{-1/100} S$.  By \eqref{xfx} and \eqref{ffn}, any such interval $K_m$ will be contained in $\N_{r_{n_1}}( E_\omega )$.  Thus we see that
$$\N_{r_{n_1}}( E_\omega ) \gtrsim N^{-1/100} S 100 N^{1/100} r_{n_1 + N^{-9/100} N} \gtrsim r_{n_1 - N^{-9/100} N}.$$
But this contradicts \eqref{smallproj} if $N$ is large enough.  This establishes Theorem \ref{twoproj-quant}.

\section{Reduction to a normal component}\label{norm-sec}

We now begin the proof of Theorem \ref{main}.  The main idea is to split $E \times S^1$ into a ``non-normal'' region, which roughly speaking behaves somewhat like a Lipschitz graph and so can be controlled by the unrectifiability hypothesis, and a ``normal'' region which will be dealt with by variants of the Lebesgue differentiation theorem.  In this section we handle the non-normal component, leaving the normal component for later.

Let $\mu$ be 1-dimensional spherical measure $\S^1$ restricted to $E$, thus $\mu$ is supported on $E$ and $\mu(E) \leq L$.  We let $\mu \times \sigma$ be the product measure on $E \times S^1$, thus
\begin{equation}\label{muse}
\mu \times \sigma( E \times S^1) \leq L.
\end{equation}

The following basic lemma will be useful for eliminating several error terms:

\begin{lemma}\label{aproj} For any $\mu \times \sigma$-measurable $A \subset E \times S^1$ we have
$$ \Fav(A) \lesssim \mu \times \sigma(A).$$
\end{lemma}

\begin{proof}  By Fubini's theorem, it suffices to show that $m( \{ x \cdot \omega: (x,\omega) \in A \} ) \lesssim \mu( x: (x,\omega) \in A \}$ for almost every $\omega \in S^1$.  But the map $x \mapsto x \cdot \omega$ is a contraction.  Since $\mu$ is the restriction of $1$-dimensional spherical measure, the claim follows.
\end{proof}

Lemma \ref{aproj} already yields the claim when $N$ is bounded, so without loss of generality we can assume $N$ to be large, say $N \geq 10^{100}$.  For similar reasons we can also take $N^{1/100}$ to be an integer.  

Given any $(x,\omega) \in E \times S^1$, $r > 0$, and $M > 0$, we define the (double) sector
$$ X_{(x,\omega)}(r,M) := \{ y \in \R^2: |y-x| < r; |(y-x) \cdot \omega| \leq \frac{1}{M} |y-x| \}$$
(see Figure \ref{fig2}.)

\begin{figure}[tb]
\centerline{\psfig{figure=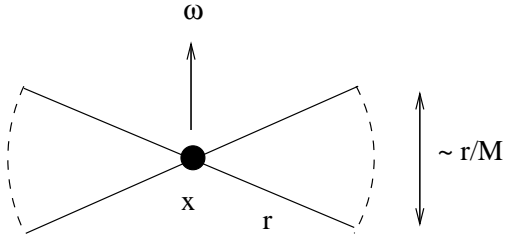}}
\caption{A sector $X_{(x,\omega)}(r,M)$.}
\label{fig2}
\end{figure}

Note that this sector has a ``thickness'' comparable to $r/M$ in the $\omega$ direction, and so we expect the measure $\mu(X_{x,\omega}(r,M))$ to exceed this quantity when $\omega$ is somehow ``normal'' to $E$.  Let us formalise this with a definition:

\begin{definition}[Normal direction]\label{norm-def}  Let $100 < n \leq N-100$ and let $M > 10^5$.  We say that a pair $(x,\omega) \in E \times S^1$ is \emph{normal at scale $n$} with Lipschitz constant $M$ if we have
\begin{equation}\label{muxri}
 \mu( X_{x,\omega}(r, M/10^4) \backslash X_{x,\omega}(r_{n+100,-}, M/10^4) ) > N^{-1/100} r/M
\end{equation}
for at least one $r_{n+100,-} \leq r \leq r_{n-100,+}$.  We let $\Norm_{n,M} \subset E \times S^1$ denote the set of all pairs which are normal at scale $n$ and Lipschitz constant $M$.  (See Figure \ref{fig3}.)
\end{definition}

\begin{figure}[tb]
\centerline{\psfig{figure=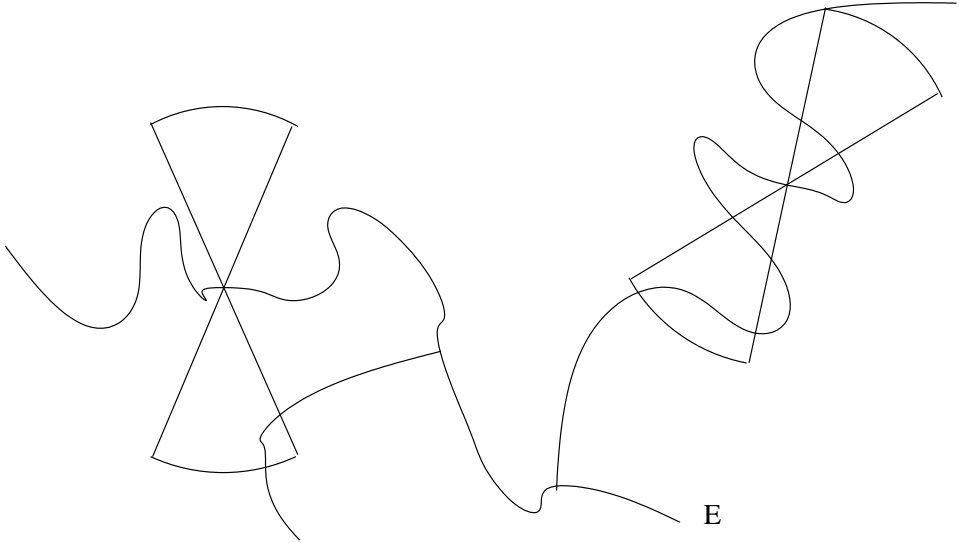}}
\caption{The sectors associated to a non-normal pair (left) and a normal pair (right).  In this figure $E$ is depicted as connected, but this is of course not the case in general; for instance one could imagine replacing $E$ here by a moderately dense subset.}
\label{fig3}
\end{figure}

It is easy to verify that $\Norm_{n,M}$ is measurable with respect to $\mu \times \sigma$. 
In the next section we shall establish the following somewhat technical result.

\begin{proposition}[Normal directions negligible]\label{tan} Let the notation and hypotheses be as in Theorem \ref{main}.  There exists a scale index $100 < n \leq N-100$ and a Lipschitz constant $M \leq 1/r_{n-100,-}$ such that
$$ \Fav( \Norm_{n,M}(E) ) \lesssim N^{-1/100} L.$$
\end{proposition}

Let us assume Proposition \ref{tan} for now, and see how it implies Theorem \ref{main}.  If we let $G$ be the space of ``good'' or ``non-normal'' directions
$$ G := (E \times S^1) \backslash \Norm_{n,M}$$
then by \eqref{fav-sub} it suffices to show that
$$ \Fav( G ) \lesssim N^{-1/100} L.$$
For each $\omega$, let $G_\omega \subset E$ denote the set
$$ G_\omega := \{ x: (x,\omega) \in G \};$$
by Fubini's theorem, it thus suffices to show that
\begin{equation}\label{gom}
 m( \{ x \cdot \omega: x \in G_\omega \} ) \lesssim N^{-1/100} L
\end{equation}
for all $\omega \in S^1$.

Fix $\omega$.  We first get rid of an exceptional set.  By \eqref{length-bound} we can find a finite collection $\B_n$ of open balls with radius between $r_{n,-}$ and $r_{n,+}$ which cover $E$, such that \begin{equation}\label{lb}
\sum_{B \in \B_n} r(B) \lesssim L.  
\end{equation}
If $B \in \B_n$ is such a ball, we say that an interval $J \subset \R$ is \emph{low-density relative to $B$ and $\omega$} if $|J| \leq r_{n,-}$ and
\begin{equation}\label{lowdens}
\mu( \{ x \in B \cap G_\omega: x \cdot \omega \in 5J \} ) \leq 10^{10} N^{-1/100} |J|,
\end{equation}
where $5J$ is the interval with the same centre as $J$ but five times the length.
We let $E_\omega$ denote the union of all the sets $\{ x \in B \cap G_\omega: x \cdot \omega \in J \}$ generated by balls $B \in \B_n$ and intervals $J$ which are low-density relative to $B$ and $\omega$.

\begin{lemma}  We have $\mu(E_\omega) \lesssim N^{-1/100} L$.
\end{lemma}

\begin{proof}  By \eqref{length-bound} it suffices to show that
$$ \mu( \bigcup_{J \in {\mathcal J}_B}\{ x \in B \cap G_\omega: x \cdot \omega \in J \} ) \lesssim N^{-1/100} r(B)$$
for all $B \in \B_n$, where ${\mathcal J}_B$ is the collection of all the intervals $J$ which are low-density relative to $B$ and $\omega$.  By monotone convergence (and the separability of $\R$), it suffices to show this with ${\mathcal J}_B$ replaced by any finite subcollection of such intervals.  We can also clearly restrict attention to those $J$ which intersect $\{ x \cdot \omega^\perp: x \in B \}$. Using Wiener's Vitali-type covering lemma\footnote{This lemma asserts that given any finite collection of balls $B_1,\ldots,B_n$, one can find a subcollection $B_{i_1},\ldots,B_{i_K}$ of disjoint balls such that the dilates $5B_{i_1},\ldots,5B_{i_K}$ cover $\bigcup_{j=1}^n B_j$.}, we can then cover this subcollection by $5J_1,\ldots,5J_K$ for some disjoint $J_1,\ldots,J_K$ in ${\mathcal J}_B$ which intersect $\{ x \cdot \omega^\perp: x \in B \}$.  But by hypothesis \eqref{lowdens} we have
$$ \mu( \{ x \in B \cap G_\omega: x \cdot \omega \in 5J_k \} ) \lesssim N^{-1/100} |J_k|$$
for all $1 \leq k \leq K$.  Since the $J_k$ are disjoint, have length at most $|J_k| \leq r_{n,-} \leq r(B)$, and intersect $\{ x \cdot \omega: x \in B \}$, we see that $\sum_{k=1}^K |J_k| \lesssim r(B)$.  Summing the previous estimate in $k$ we obtain the claim.
\end{proof}

We now use rotational symmetry to normalise $\omega = e_1$.
Let $G'_{\omega} := G_{\omega} \backslash E_\omega$.
In view of the above lemma, we see that to show \eqref{gom} it suffices to show that
\begin{equation}\label{rect}
 m( \{ x_1 \in \R: (x_1,x_2) \in G'_{\omega} \hbox{ for some } x_2 \in \R \} ) \lesssim N^{-1/100} L.
 \end{equation}

We now make a key geometric observation (cf. \cite[Lemma 15.14]{mattila}), that $G'_{\omega}$ is behaving very much like a Lipschitz graph:

\begin{lemma}[Approximate Lipschitz property]\label{alip}  Let $(x_1,x_2), (y_1,y_2) \in G'_{\omega}$ be such that $|x_1 - y_1| + |x_2-y_2| \leq 10 r_{n-1,-}$.  Then $|x_2-y_2| \leq \frac{1}{10} M |x_1-y_1| + \frac{1}{10} r_{n+2,+}$.  (See Figure \ref{fig4}.)
\end{lemma}

\begin{proof}  Fix $(x_1,x_2) \in G'_\omega$, and define the set
$$ H := \{ (y_1,y_2) \in G_\omega: |x_1 - y_1| + |x_2-y_2| \leq 10 r_{n-1,-} \}$$
and the quantity
$$ R := \sup \{ |x_2-y_2|: (y_1,y_2) \in H \}.$$
Clearly $R \leq 10 r_{n-1,-}$.  To establish the claim, it will suffice (since $G_\omega$ contains $G'_\omega$) to show that $R \leq \frac{1}{10} M |x_1-y_1| + \frac{1}{10} r_{n+2,+}$.  Suppose this is not the case, thus
\begin{equation}\label{yx2}
\frac{1}{10} M |x_1-y_1| + \frac{1}{10} r_{n+2,+} < R \leq 10 r_{n-1,-}.
\end{equation}
By definition of $R$, we can thus find $(y_1,y_2) \in H$ such that
\begin{equation}\label{r2}
 \frac{1}{2} R \leq |x_2-y_2| \leq R;
 \end{equation}
in particular, from \eqref{yx2} and \eqref{r2} we have
\begin{equation}\label{y3}
|x_2-y_2| \geq \frac{1}{20} M |x_1-y_1|, \frac{1}{20} r_{n+2,+}.
\end{equation}
Fix this $(y_1,y_2)$.  By the non-normality hypothesis, we have
$$ \mu( (X_{(x_1,x_2,\omega)}(r, M/10^4) \backslash X_{(x_1,x_2),\omega}(r_{n+100,-}, M/10^4)) \cap G_\omega ) \leq N^{-1/100} r/M $$
and
$$ \mu( (X_{(y_1,y_2,\omega)}(r, M/10^4) \backslash X_{(y_1,y_2),\omega}(r_{n+100,-}, M/10^4)) \cap G_\omega ) \leq N^{-1/100} r/M $$
for all $r_{n+100,-} \leq r \leq r_{n-100,+}$, and thus
\begin{equation}\label{con1}
 \mu( Y_r \cap G_\omega ) \leq 2 N^{-1/100} r/M
 \end{equation}
where
\begin{align*}
 Y_r &:= (X_{(x_1,x_2,\omega)}(r, M/10^4) \backslash X_{(x_1,x_2),\omega}(r_{n+100,-}, M/10^4)) \\
&\quad \cup (X_{(y_1,y_2,\omega)}(r, M/10^4) \backslash X_{(y_1,y_2),\omega}(r_{n+100,-}, M/10^4)).
\end{align*}
We apply this fact with $r := 1000 |y_2-x_2|$; note that \eqref{r2}, \eqref{yx2}, \eqref{scale-sep} clearly ensure that $r$ lies in the range $r_{n+100,-} \leq r \leq r_{n-100,+}$.  

Let $J$ be the interval $J := [\frac{x_1+y_1}{2} - \frac{10|x_2-y_2|}{M},\frac{x_1+y_1}{2} + \frac{10|x_2-y_2|}{M}]$.  From \eqref{y3} we see that $[x_1,y_1] \subseteq J$. We claim the set inclusion
\begin{equation}\label{con2}
\{ (z_1,z_2) \in H: z_1 \in 5J \} \subset Y_r.
\end{equation}
Indeed, if $(z_1,z_2) \in H$ is such that $z_1 \in 5J$, then 
$$|z_2-x_2| \leq R \leq 2 |x_2-y_2| \hbox{ and } |z_1-\frac{x_1+y_1}{2}| \leq \frac{50}{M} |x_2-y_2|,$$
and thus by the triangle inequality and \eqref{y3}
\begin{equation}\label{zoo}
|z_2-x_2|, |z_2-y_2| \leq 3 |x_2-y_2| \hbox{ and } |z_1-x_1|, |z_1-y_1| \leq \frac{60}{M} |x_2-y_2|.
\end{equation}
By the triangle inequality, we either have $|z_2 - x_2| \geq \frac{1}{2} |y_2-x_2|$ or $|z_2 - y_2| \geq \frac{1}{2} |y_2-x_2|$.  If $|z_2 - y_2| \geq \frac{1}{2} |y_2-x_2|$ we see from \eqref{y3}, \eqref{zoo} that
$$ |z_2 - y_2| \geq \frac{1}{2} |y_2-x_2| \geq \frac{1}{120} M |z_1-y_1|$$
which implies from choice of $r$ that 
$$(z_1,z_2) \in X_{(y_1,y_2,\omega)}(r,M/10^4) \backslash X_{(y_1,y_2),\omega}(r_{n+100,-}, M/10^4).$$
  Similarly, if $|z_2 - x_2| \geq \frac{1}{2} |y_2-x_2|$ then 
$$(z_1,z_2) \in X_{(x_1,x_2,\omega)}(r,M/10^4) \backslash X_{(x_1,x_2),\omega}(r_{n+100,-}, M/10^4).$$
In either case we obtain \eqref{con2}.

The point $(x_1,x_2)$ lies in $E$, and is thus contained in a ball $B$ in $\B_n$, whose radius $r(B)$ is at most $r_{n,+}$.  In particular we see from definition of $H$ (and \eqref{scale-sep}) that $B \cap G_\omega \subset H$.  Thus from \eqref{con2} and \eqref{con1} we have
$$ \mu( \{ (z_1,z_2) \in B \cap G_\omega: z_1 \in 5J \} ) \leq 10^5  N^{-1/100} |J| / M$$
(say).  Comparing this with \eqref{lowdens} we see that $J$ is low-density relative to $B$ and $\omega$, and thus $(x_1,x_2) \in E_\omega$, contradicting the hypothesis that $(x_1,x_2) \in G'_\omega$.
\end{proof}

\begin{figure}[tb]
\centerline{\psfig{figure=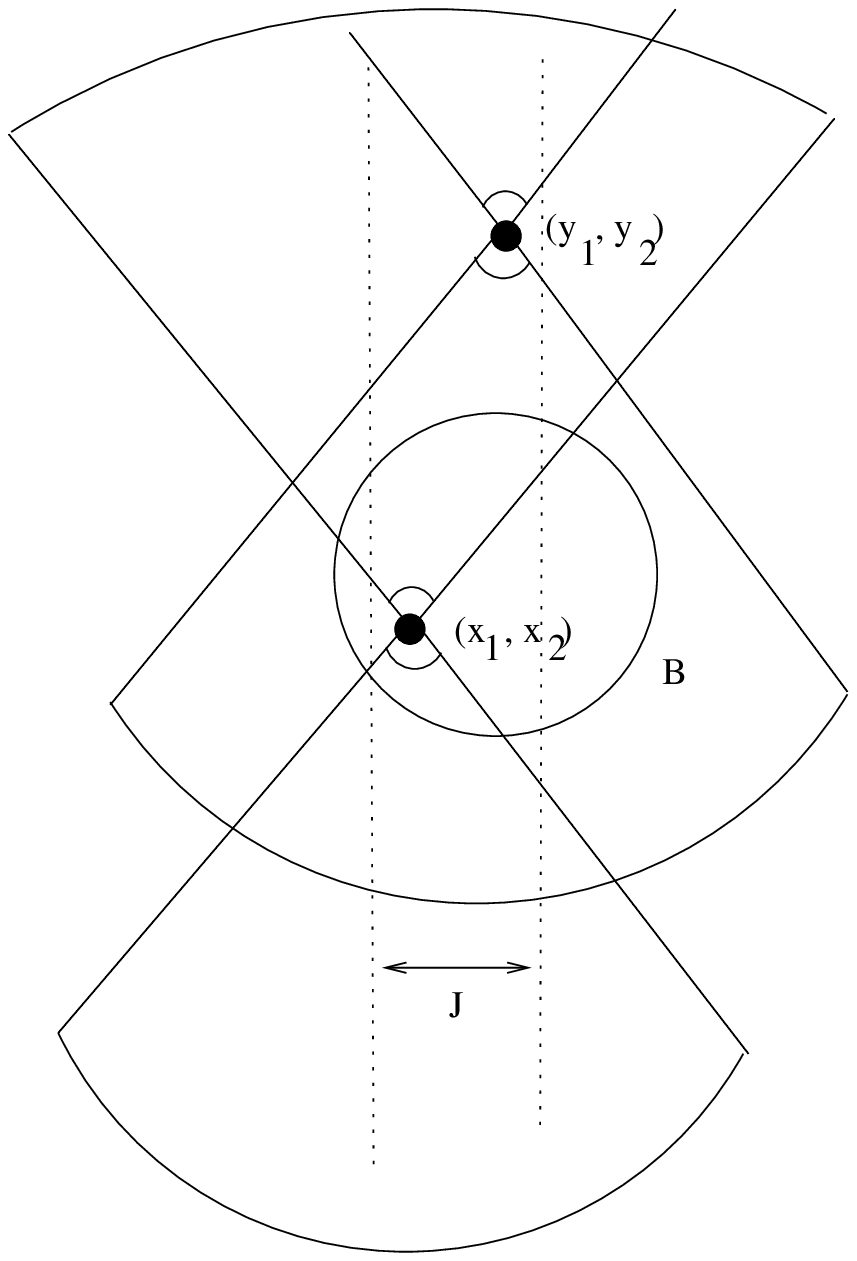}}
\caption{The situation in the proof of Lemma \ref{alip}.}
\label{fig4}
\end{figure}

We now modify the standard proof of the well-known result that a partially defined real-valued Lipschitz function extends to a totally defined real-valued Lipschitz function to obtain:

\begin{corollary}[Explicit rectifiability]  Let $B \subset \R^2$ be a ball of radius at most $r_{n,-}$.  Then there exists a Lipschitz function $F_B: \R \to \R$ of Lipschitz constant at most $M$ such that $|x_2 - F_B(x_1)| \leq r_{n+2,+}$ for all $(x_1,x_2) \in G'_{\omega} \cap B$.
\end{corollary}

\begin{proof}  We may assume that $G'_{\omega} \cap B$ is non-empty, otherwise there is nothing to prove.  If we then define
$$ F(x) := \sup \{ y_2 + M |x-y_1|: (y_1,y_2) \in G'_{\omega} \}$$
then the claim easily follows from the previous lemma.
\end{proof}

From this and Definition \ref{rec} we see that
$$ m( \{ x_1 \in \R: (x_1,x_2) \in G'_{\omega} \cap B \hbox{ for some } x_2 \in \R \} ) \lesssim R_E(r_{n+2,+}, M) r(B)$$
for all $B \in \B_{n}$. Summing in $B$ using \eqref{lb} we obtain
$$ m( \{ x_1 \in \R: (x_1,x_2) \in G'_{\omega} \cap B \hbox{ for some } x_2 \in \R \} ) \lesssim R_E(r_{n+2,+}, M) L$$
and the claim \eqref{rect} follows from \eqref{re}.  This concludes the proof of Theorem \ref{main} assuming Proposition \ref{tan}.

\section{Eliminating high multiplicity directions}\label{highmult}

We now locate some high multiplicity directions of $E$, and show that their direct contribution is negligible.  A key new difficulty in the quantitative setting is that the notion of high multiplicity will depend on scale, and we will need the pigeonhole principle (Lemma \ref{pigeon}) in order to find a favourable scale with which to perform the analysis.

\begin{definition}[High multiplicity lines]\label{hidef}  Let $1 \leq n \leq N$.  A line $l \subset \R^2$ is said to have \emph{high multiplicity at a scale index $\leq n$} if the set $E \cap l$ contains a set of cardinality at least $N^{1/100}$ which is $r_{n,-}$-separated, thus any two points in $E \cap l$ are separated by a distance of at least $r_{n,-}$.  We let $H_n \subset E \times S^1$ to be the set of all points $(x,\omega) \in E \times S^1$ such that the line $l_{x \cdot \omega,\omega}$ is high multiplicity at a scale index $\leq n$.
\end{definition}

Because $E$ is compact, it is not difficult to show that $H_n$ is also compact.  Also, from \eqref{scale-mono} we clearly have the nesting property 
$$ H_1 \subset H_2 \subset \ldots \subset H_N \subset E \times S^1.$$
Applying \eqref{muse} and Lemma \ref{pigeon}, we may thus find a scale index $0.1 N < n_0 \leq 0.9 N$ near which the $H_n$ are stable in the sense that
\begin{equation}\label{mus}
\mu \times \sigma( H_{n_0 + N^{-3/100} N} \backslash H_{n_0 - N^{-3/100} N} ) \lesssim N^{-3/100} L.
\end{equation}

We now fix this $n_0$ and work entirely within the scale indices 
$$n \in [n_0-N^{-3/100} N, n_0+N^{-3/100} N] \subset [1,N].$$

Now we show that near $N_0$, the $H_n$ have negligible Favard length.  In fact we shall show something slightly stronger:

\begin{definition}[Angular neighbourhoods]  Let $F \subset E \times S^1$ and $\theta > 0$.  We define the \emph{$\theta$-angular neighbourhood} $\Nb_\theta(F)$ of $F$ to be the set
$$ \Nb_\theta(F) := \{ (x,\omega) \in E \times S^1: \angle(\omega,\omega') < \theta \hbox{ for some } (x,\omega') \in F \}.$$
\end{definition}

We define the set
\begin{equation}\label{thdef}
 \tilde H := \Nb_{r_{n_0-N^{-3/100} N + 10,-}}(H_{n_0-N^{-3/100} N})
 \end{equation}
and the slightly smaller set
\begin{equation}\label{hpdef}
 H' := \Nb_{\frac{1}{2} r_{n_0-N^{-3/100} N + 10,-}}(H_{n_0-N^{-3/100} N}).
\end{equation}

\begin{lemma}[$\tilde H$ Favard-negligible]\label{hnfn}  We have
$$ \Fav(\tilde H) \lesssim N^{-1/100} L.$$
\end{lemma}

\begin{remark} This is a quantitative counterpart of \cite[Lemma 18.4]{mattila}.
\end{remark}

\begin{proof}  Let us write $n^-_0 := n_0 - N^{-3/100} N$ for brevity. From \eqref{length-bound} we can find a collection $\B$ of balls $B$ with radius $r(B) \in [r_{n^-_0+5,-},r_{n^-_0+5,+}]$ which cover $E$ such that 
\begin{equation}\label{sumb}
\sum_B r(B) \lesssim L.
\end{equation}
If a line $l_{x \cdot \omega, \omega}$ is high multiplicity at a scale index $\leq n^-_0$, then by definition it intersects $E$ in $N^{1/100}$ points $x_1,\ldots,x_{N^{1/100}}$ which are $r_{n^-_0,-}$-separated.  Each of these points is contained in a ball $B_1,\ldots,B_{N^{1/100}}$ in $\B$, where each $B_i$ has radius $r(B_i) \leq r_{n^-_0+5,+} \leq \frac{1}{32} r_{n^-_0,-}$, thanks to \eqref{scale-mono}, \eqref{scale-sep}; in particular, the $B_i$ are all disjoint.  Now let $\omega' \in S^1$ be such that $\angle(\omega,\omega') \leq r_{n^-_0-10,-}$, thus $\angle(\omega,\omega') \leq \frac{1}{32} r(B_i)$ by \eqref{scale-mono}, \eqref{scale-sep}.
From elementary geometry (and the fact that $x, x_1,\ldots,x_{N^{1/100}}$ all lie in $B(0,1)$) we then see that the line $l_{x \cdot \omega', \omega'}$ then meets each of the dilated balls $5B_i$ (defined as the ball with the same centre as $B_i$ and five times the radius) in a line segment of length $\gtrsim r(B_i)$ (see Figure \ref{fig5}). To put it another way, we have
$$ \int_{l_{x \cdot \omega', \omega'}} \frac{1}{r(B_i)} 1_{B_i}\ d\S^1 \gtrsim 1$$
for all $1 \leq i \leq N^{1/100}$.  Summing this, we conclude that
$$ \int_{l_{x \cdot \omega', \omega'}} \sum_{B\in \B} \frac{1}{r(B)} 1_{B}\ d\S^1 \gtrsim N^{1/100}$$
for all $(x,\omega') \in \tilde H$.  By Fubini's theorem, we conclude that for every $\omega' \in S^1$ we have
$$ \int_{\R^2} \sum_{B\in \B} \frac{1}{r(B)} 1_{B}\ d\S^1 \gtrsim N^{1/100} m(\{ x \cdot \omega': (x,\omega') \in \tilde H \}).$$
Applying \eqref{sumb}, we obtain
$$ m(\{ x \cdot \omega': (x,\omega') \in \tilde H \}) \lesssim N^{-1/100} L.$$
Integrating this in $\omega'$, we obtain the claim.
\end{proof}

\begin{figure}[tb]
\centerline{\psfig{figure=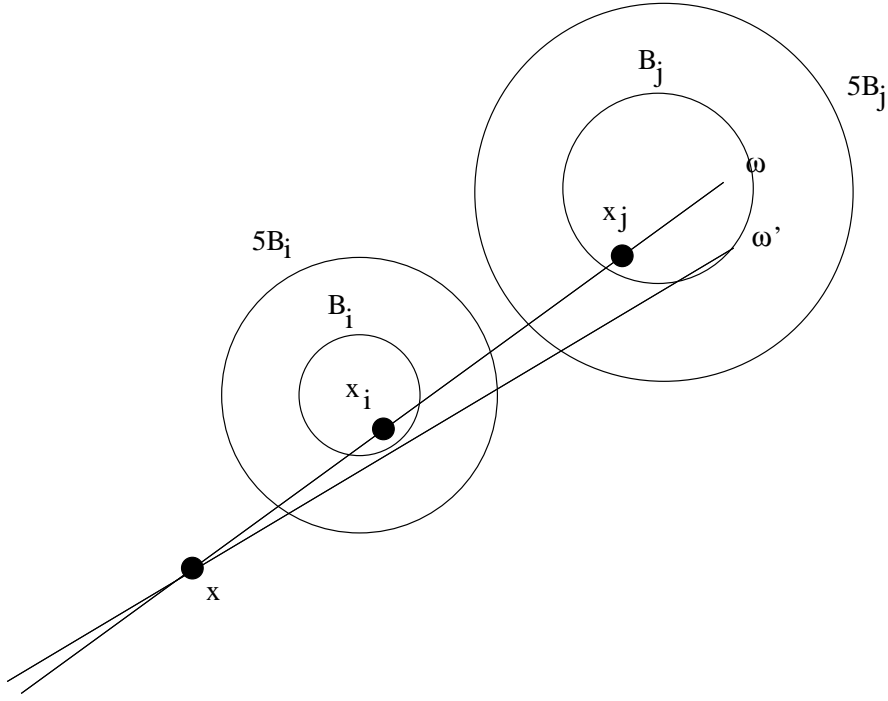}}
\caption{The situation in the proof of Lemma \ref{hnfn}, not drawn to scale.}
\label{fig5}
\end{figure}

We define the exceptional set 
$$\Delta H := H_{n_0+N^{-3/100} N} \backslash H',$$
thus from \eqref{mus} we have that
\begin{equation}\label{mus2}
\mu \times \sigma( \Delta H ) \lesssim N^{-3/100} L.
\end{equation}
Because of this, the contribution of $\Delta H$ will be manageable in the sequel.  We make the technical remark that $\Delta H$ is angularly closed, i.e. for any $x \in E$, the set $\{ \omega: (x,\omega) \in \Delta H \}$ is closed.

\section{Eliminating positive multiplicity directions}\label{posmult}

In the preceding section we obtained two sets $\tilde H, \Delta H \subset E \times S^1$ of high multiplicity which were well controlled.  Now we perform another scale refinement to also control those directions in which the multiplicity is merely \emph{positive}, leaving only the zero-multiplicity directions.

\begin{definition}[Positive multiplicity direction]\label{posdef}  Let $n_0 - 0.9 N^{-3/100} N \leq n \leq n_0 + 0.9 N^{-3/100} N$.  We say that a point $(x,\omega) \in E \times S^1$ has \emph{positive multiplicity at scale index $n$} if the line $l_{x \cdot \omega,\omega}$ contains a point $y$ in $E$ such that $r_{n-N^{-7/100} N,-} \leq |x-y| \leq r_{n+N^{-7/100} N,+}$.  We let $P_n \subset E \times S^1$ be the set of all $(x,\omega)$ which have positive multiplicity at scale index $n$.
\end{definition}

Suppose that $(x,\omega)$ does not lie in $H' \cup \Delta H$, so in particular $(x,\omega)$ does not lie in $H_{n_0+N^{-3/100} N}$.  Applying Definition \ref{hidef}, we conclude that $l_{x \cdot \omega, \omega}$ contains at most $N^{1/100}$ points of $E$ which are $r_{n_0+N^{-3/100} N,-}$-separated.  Comparing this with Definition \ref{posdef}, we see that $(x,\omega)$ can lie in at most $O(N^{1/100} N^{-7/100} N)$ of the sets $P_n$.  Integrating this and using \eqref{muse}, we conclude that
$$ \sum_{n_0 - 0.9 N^{-3/100} N \leq n \leq n_0 + 0.9 N^{-3/100} N} \mu \times \sigma(P_n \backslash (\tilde H \cup \Delta H)) \lesssim N^{1/100} N^{-7/100} N L.$$
Applying the pigeonhole principle, we can thus find
$$ n_0 - 0.9 N^{-3/100} N \leq n_1 \leq n_0 + 0.9 N^{-3/100} N$$
such that
\begin{equation}\label{mupn}
 \mu \times \sigma(P_{n_1} \backslash (H' \cup \Delta H)) \lesssim N^{-3/100} L.
\end{equation}

Because of this, the contribution of $P_{n_1}$ will be manageable in the sequel.  Also observe that $P_{n_1}$ is closed, and hence $P_{n_1} \backslash H'$ is angularly closed in the sense of the preceding section.
Henceforth we fix $n_1$, and work entirely within the scale indices 
$$ n \in [n_1 - N^{-7/100} N, n_1+N^{-7/100} N] \subset [n_0-N^{-3/100} N, n_0+N^{-3/100} N] \subset [1,N].$$

\section{Eliminating high density strips}\label{highdens}

We now perform a variant of the analysis of the preceding sections, in order to eliminate the contribution of certain strips $\{ x: x \cdot \omega \in J \}$ in $\R^2$ which are capturing too much of the mass of $\mu$.

\begin{definition}[High density strips]  Let $1 \leq n \leq N$.  An infinite strip $\{ x \in \R^2: x \cdot \omega \in J \}$, where $\omega \in S^1$ and $J$ is an interval, is said to have \emph{high density} if we have
$$ \mu( \{ x \in E: x \cdot \omega \in J \} ) \geq N^{1/100} m(J).$$
We let $D_n \subset E \times S^1$ be the set of all points $(x,\omega) \in E \times S^1$ such that the line $l_{x \cdot \omega,\omega}$ lies in a high-density strip $\{ x \in \R^2: x \cdot \omega \in J \}$ whose width $m(J)$ is at least $r_{n,-}$.
\end{definition}

One easily verifies that $D_n$ is compact and that
$$ D_1 \subset D_2 \subset \ldots \subset D_N \subset E \times S^1.$$
Restricting to scale indices $n$ between $n_1 - 0.9 N^{-7/100} N$ and $n_1 + 0.9 N^{-7/100} N$, and then 
applying \eqref{muse} and Lemma \ref{pigeon}, we can locate a scale index $n_1 - 0.9 N^{-7/100} N \leq n_2 \leq n_1 + 0.9 N^{-7/100} N$ such that
\begin{equation}\label{mud}
 \mu \times \sigma( D_{n_2 + N^{-10/100} N} \backslash D_{n_2 - N^{-10/100} N} ) \lesssim N^{-3/100} L.
\end{equation}
Let us now fix this $n_2$, and work entirely inside the range of scale indices
\begin{align*}
n &\in [n_2 - N^{-10/100} N, n_2+N^{-10/100} N] \\
&\subset [n_1 - N^{-7/100} N, n_1+N^{-7/100} N] \\
&\subset [n_0-N^{-3/100} N, n_0+N^{-3/100} N] \\
\subset [1,N].
\end{align*}

Define
$$ \tilde D := \N_{r_{n_2-N^{-10/100} N+10,-}}( D_{n_2-N^{-10/100} N} )$$
and the slightly smaller set
$$ D' := \N_{\frac{1}{2} r_{n_2-N^{-10/100} N+10,-}}( D_{n_2-N^{-10/100} N} )$$
We have an analogue of Lemma \ref{hnfn}:

\begin{lemma}[$\tilde D$ Favard-negligible]\label{dnfn} We have
$$ \Fav(\tilde D) \lesssim N^{-1/100} L.$$
\end{lemma}

\begin{proof}
By Fubini's theorem, it suffices to show that
$$ m( \{ x \cdot \omega: (x,\omega) \in \tilde D \} ) \lesssim N^{-1/100} L$$
for all $\omega \in S^1$.

Fix $\omega$; by rotation symmetry we can take $\omega = e_1$, so we need to show
\begin{equation}\label{mool}
m( \{ x_1: (x_1,x_2,e_1) \in \tilde D \} ) \lesssim N^{-1/100} L.
\end{equation}
We also abbreviate $n_2^- := n_2-N^{-10/100} N$.  Unwrapping the definitions, we see that if $(x_1,x_2,e_1) \in \tilde D$, then there exists $\omega' \in S^1$ with $\angle(\omega', e_1) \leq r_{n_2^- + 10,-}$ and an interval $J$ containing $x \cdot \omega'$ with $m(J) \geq r_{n_2^-,-}$ such that
$$ \mu( \{ (y_1,y_2) \in E: (y_1,y_2) \cdot \omega \in J \} ) \geq N^{1/100} m(J).$$
From \eqref{scale-mono}, \eqref{scale-sep} we have $\angle(\omega', e_1) \leq \frac{1}{2^{10}} m(J)$.  Also,
$(x_1,x_2)$ and $(y_1,y_2)$ lie in the unit ball. From this, we see that $x_1 \in 2J$, and furthermore if $(y_1,y_2) \in E$ is such that $(y_1,y_2) \cdot \omega \in J$, then $y_1 \in 2J$.  Thus, if we let $\mu_1$ be the pushforward of the measure $\mu$ to $\R$ under the projection map $(y_1,y_2) \mapsto y_1$, we see that
$$ \mu_1(2J) \geq N^{1/100} m(J).$$
Since $2J$ contains $x_1$, we thus see that the Hardy-Littlewood maximal function $M\mu_1(x) := \sup_{r > 0} \frac{1}{2r} \mu_1([x-r,x+r])$ of $\mu_1$ takes the value $\gtrsim N^{1/100}$ on $x_1$.  On the other hand, from \eqref{muse} we see that the total mass $\mu_1(\R)$ of the positive measure $\mu_1$ is at most $L$.  The claim \eqref{mool} then follows from the Hardy-Littlewood maximal inequality $m( \{ x: M \mu \geq \lambda \} ) \lesssim \frac{1}{\lambda} \mu(\R)$ for measures (see e.g. \cite[Theorem 2.19]{mattila}).
\end{proof}

If we define
$$ \Delta D := D_{n_2 + N^{-10/100} N} \backslash D'$$
then from \eqref{mud} we have
\begin{equation}\label{mud-2}
\mu \times \sigma( \Delta D ) \lesssim N^{-3/100} L.
\end{equation}

We now define a unified exceptional set
$$ \Delta := \Delta H \cup (P_{n_1} \backslash H') \cup \Delta D;$$
from \eqref{mus2}, \eqref{mupn}, \eqref{mud-2} we have
\begin{equation}\label{delta-small}
\mu \times \sigma(\Delta) \lesssim N^{-3/100} L.
\end{equation}
Also, $\Delta$ is angularly closed. Because of this, the contribution of $\Delta$ will be manageable in the sequel.

\section{Conclusion of the argument}\label{conclude-sec}

Having refined our set of scales suitably, and controlled various exceptional sets, namely two large but coarse sets $\tilde H, \tilde D$, and one fine but small set $\Delta$, we are now ready to establish Proposition \ref{tan}.

It will suffice to show that
$$ \Fav( \Norm_{n_2, 10^4/r_{n_2-200,-}} ) \lesssim N^{-1/100} L.$$

In view of Lemma \ref{hnfn}, Lemma \ref{dnfn}, and \eqref{fav-sub}, it suffices to show that
\begin{equation}\label{favor}
 \Fav( F ) \lesssim N^{-1/100} L
 \end{equation}
where $F$ is the set
$$ F := \Norm_{n_2, 10^4/r_{n_2-200,-}} \backslash (\tilde H \cup \tilde D).$$

The key observation is that $F$ only consists of points which are close to many points in $\Delta$.

\begin{lemma}\label{om} Let $(x,\omega) \in F$.  Then the set
$$ \Omega_{x,\omega} := \{ \omega' \in S^1: \angle(\omega,\omega') \leq 10^4 r_{n_2-200,-}; (x,\omega') \in \Delta \}
\subset S^1$$
has measure $\sigma(\Omega_{x,\omega}) \gtrsim N^{-2/100} r_{n_2-200,-}$.
\end{lemma}

\begin{proof} From Definition \ref{norm-def} we have
\begin{equation}\label{mrr}
 \mu( X_{x,\omega}( r, 1/r_{n_2-200,-} ) \backslash X_{x,\omega}(r_{n_2+100,-}, 1/r_{n_2-200,-}) ) \gtrsim N^{-1/100} r r_{n_2-200,-} 
 \end{equation}
for some $r_{n_2+100,-} \leq r \leq r_{n_2-100,+}$.  

Fix this $r$. Since $\Delta$ is angularly closed, $\Omega_{x,\omega}$ is a compact subset of the circle $S^1$.  Thus we can find a finite number of open arcs $I_1,\ldots,I_K$ in the arc $I_* := \{ \omega' \in S^1: \angle(\omega,\omega') \leq 10^5 r_{n_2-200,-} \}$ which cover $\Omega_{x,\omega}$ and such that
\begin{equation}\label{soso}
 \sigma( \bigcup_{k=1}^K I_k ) \lesssim \sigma( \Omega_{x,\omega} ).
 \end{equation}
By enlarging these arcs slightly we can assume that each arc contains at least one point in the complement of $\Omega_{x,\omega}$; by concatenating overlapping arcs we may assume that these arcs are disjoint.  

Since $(x,\omega)$ lies in $F$, it lies outside of $\tilde H$, and thus by \eqref{hpdef}, \eqref{thdef} we have $(x,\omega') \not \in H'$ for any $\omega' \in I_*$.  A similar argument gives $(x,\omega') \not \in D'$ for any $\omega' \in I_*$.

\begin{figure}[tb]
\centerline{\psfig{figure=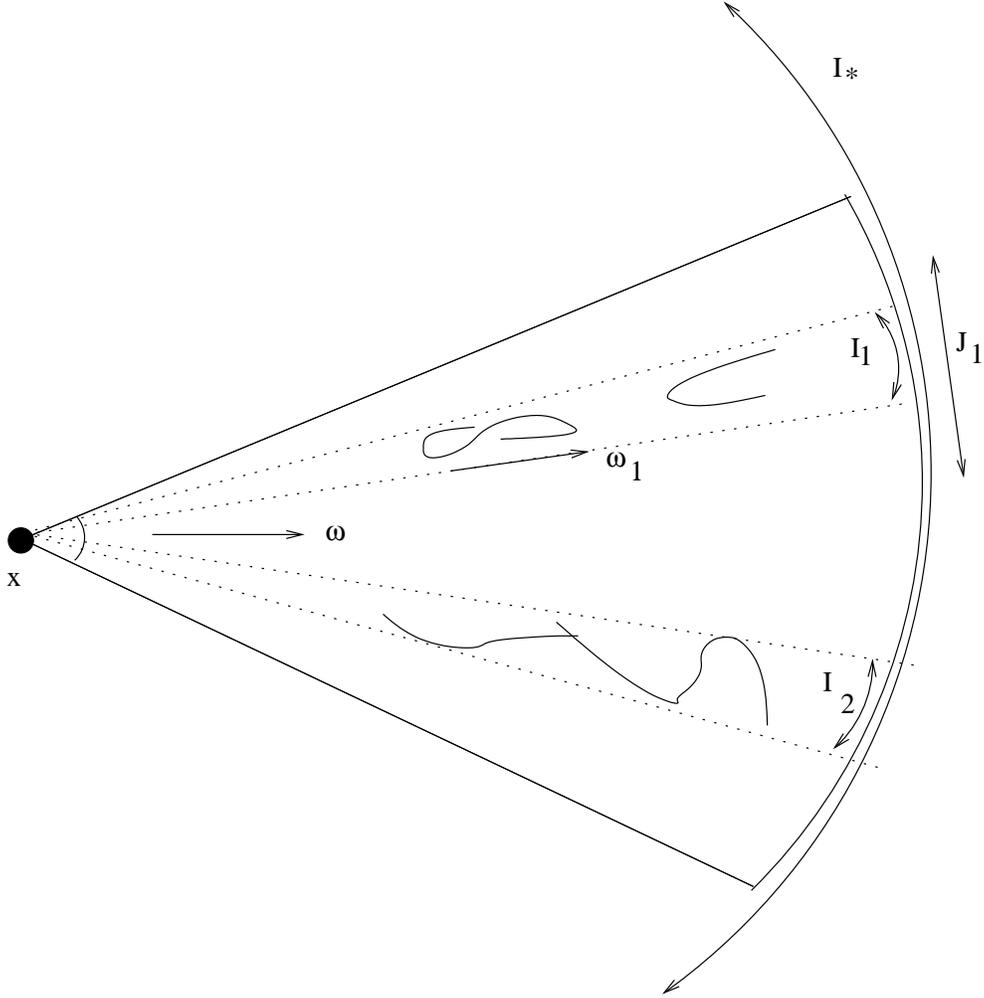}}
\caption{A typical situation in the proof of Lemma \ref{om}.}
\label{fig6}
\end{figure}

Suppose that $\omega' \in I_*$ does not lie in any of the $I_1,\ldots,I_k$, then by the above discussion $(x,\omega')$ does not lie in $\Delta$, $D'$, or $H'$.  In particular we see that $(x,\omega')$ does not lie in $P_{n_1}$.  This implies that the ray $\{ x + t \omega': r_{n_2+100,-} \leq |t| \leq r \}$ does not intersect $E$.  Since $\mu$ is supported on $E$, we conclude
$$
\mu( X_{x,\omega}( r, 1/r_{n_2-200,-} ) \backslash X_{x,\omega}(r_{n_2+100,-}, 1/r_{n_2-200,-}) )
\leq \sum_{k=1}^K \mu( \{ x + t \omega': r_{n_2+100,-} \leq |t| \leq r; \omega' \in I_k \} )$$
and thus by \eqref{mrr}
\begin{equation}\label{mrr2}
 \sum_{k=1}^K \mu( \{ x + t \omega': r_{n_2+100,-} \leq |t| \leq r; \omega' \in I_k \} ) \gtrsim
N^{-1/100} r r_{n_2-200,-}.
\end{equation}
On the other hand, given any $1 \leq k \leq K$, we have by construction that there exists at least one $\omega_k \in I_k$ which does not lie in $\Omega_{x,\omega}$.  In particular this shows that $(x,\omega_k)$ does not lie in either $D'$ or $\Delta D$, and thus does not lie in $D_{n_2 + N^{-10/100} N}$. In particular, this shows that
$$ \mu( \{ y \in E: y \cdot \omega_k \in J \} ) \leq N^{1/100} m(J)$$
whenever $J$ is an interval containing $x \cdot \omega'$ of length $m(J) \geq r_{n_2+N^{-10/100},-}$.  Now, from elementary geometry we see that the double sector $\{ x + t \omega': r_{n_2+100,-} \leq |t| \leq r; \omega' \in I_k \}$ can be contained in such a strip $\{ y \in E: y \cdot \omega_k \in J \}$ with $m(J) \sim r \sigma(I_k)$.  Thus
$$ m(\{ x + t \omega': r_{n_2+100,-} \leq |t| \leq r; \omega' \in I_k \}) \lesssim N^{1/100} r \sigma(I_k)$$
for all $1 \leq k \leq K$.  Comparing this with \eqref{mrr2}, \eqref{soso} we obtain the claim.
\end{proof}

From the above lemma, we see that if $(x,\omega)$ lies in $F$, then $\omega$ lies in the set where the Hardy-Littlewood maximal function of the indicator function of $\{ \omega' \in S^1: (x,\omega') \in \Delta \}$ on the circle is $\gtrsim N^{-2/100}$.  Applying the Hardy-Littlewood maximal inequality, we conclude that
$$ \sigma( \{ \omega: (x,\omega) \in F \} ) \lesssim N^{2/100} \sigma( \omega' \in S^1: (x,\omega') \in \Delta \}$$
for all $x \in E$.  Integrating this in $E$ we obtain
$$ \mu \times \sigma(F) \lesssim N^{2/100} \mu \times \sigma(\Delta)$$
and the claim \eqref{favor} then follows from Lemma \ref{aproj} and \eqref{delta-small}.  This concludes the proof of Theorem \ref{main}.

\appendix

\section{$\beta$-numbers and quantitative unrectifiability}

The purpose of this appendix is to use the theory of $\beta$-numbers to improve the bounds in Proposition \ref{rector}.

Given any dyadic square $Q$ and compact set $K \subset \R^2$, let $\beta_K(Q)$ denote the quantity
$$ \beta_K(Q) := \frac{2}{\diam(Q)} \inf_l \sup_{x \in K \cap Q} \dist(x, l)$$
where $l$ ranges over all lines.  Thus for instance $\beta_K(Q) = 0$ when $Q \cap K = \emptyset$, and also $\beta_K(Q) \leq \beta_{K'}(Q)$ whenever $K \subset K'$.  In \cite{jones} it was shown that
$$ \sum_Q \beta_K(3Q)^2 l(Q) \lesssim \S^1(K)$$
whenever $K$ is compact and connected, $Q$ ranges over all dyadic squares, $l(Q)$ is the length of $Q$, and $3Q$ is the square with the same centre as $Q$ but three times the sidelength.  In particular, for a Lipschitz graph $\Gamma := \{ x \omega_1 + F(x) \omega_2: -10 \leq x \leq 10 \}$ with $\|F\|_{\Lip} \leq M$ we see that
\begin{equation}\label{sumq}
 \sum_Q \beta_\Gamma(3Q)^2 l(Q) \lesssim 1+M
\end{equation}

\begin{proposition}[Rectifiability bound for product Cantor set, II]\label{rector2}  If $n \geq m > l \geq 0$ then
$$ R_{K_n \times K_n}( 2^{-m}, 2^{-l}, M ) \lesssim (1+M) / (m-l).$$
\end{proposition}

\begin{proof} Once again we can rescale $l=0$.  By enlarging $K_n \times K_n$ to $K_m \times K_m$ we can also assume $n=m$.  Let $\Gamma := \{ x \omega_1 + F(x) \omega_2: -10 \leq x \leq 10 \}$ be a Lipschitz graph with $\|F\|_{\Lip} \leq M$, and let $E := (K_n \times K_n) \cap \Gamma$, thus
$$ E = \{ x \omega_1 + F(x) \omega_2: x \in A \}$$
for some closed set $A \subset \R$.  It then suffices to show that $m(A) \lesssim (1+M) / n$.  We can assume that $m$ is large, since the claim is trivial otherwise.  In view of \eqref{sumq}, it thus suffices to show that
$$ 1 + \sum_Q \beta_E(3Q)^2 l(Q) \gtrsim n m(A).$$

Let $1 \leq j \leq n$.  Then $K_j \times K_j$ consists of $4^j$ dyadic squares of sidelength $4^{-j}$; let $E_j$ be the union of all such squares which intersect $E$, thus
$$ E \subset E_{n} \subset \ldots \subset E_{1}.$$ 
Let $1 \leq j \leq n-10$, let $Q_j$ be one of the squares in $E_j$, and let $Q_{j,1},\ldots,Q_{j,4^{10}}$ be the $4^{10}$ dyadic squares in $K_{j+10} \times K_{j+10}$ of sidelength $4^{-j-10}$ which are contained in $Q_j$.  From elementary geometry we see that if all $4^{10}$ of these squares lie in $E_{j+10}$, then $\beta_E(3Q_j) \geq 0.01$.  Viewing this contrapositively, we see that if $\beta_E(3Q_j) < 0.01$ then $E_j \backslash E_{j+10}$ contains at least one square of sidelength $4^{-j-10}$ contained in $Q_j$.  Summing this fact over all $Q_j$ in $E_j$, we see that
$$ \sum_{Q_j \subset E_j: l(Q_j) = 4^{-j}} l(Q_j)
\lesssim \sum_{Q_j \subset E_j: l(Q_j) = 4^{-j}} \beta_E(3Q_j)^2 l(Q_j)
+ \sum_{Q_{j+10} \subset E_j \backslash E_{j+10}: l(Q_{j+10}) = 4^{-j-10}} 4^{-j}.$$
A simple counting argument shows that
$$  \sum_{Q_j \subset E_j: l(Q_j) = 4^{-j}} l(Q_j) \gtrsim 4^n m_2(E_n)$$
where $m_2$ denotes two-dimensional Lebesgue measure.  Similarly
$$ \sum_{Q_{j+10} \subset E_j \backslash E_{j+10}: l(Q_{j+10}) = 4^{-j-10}} 4^{-j}
\lesssim 4^n m_2( (E_j \backslash E_{j+10}) \cap (K_n \times K_n ) ),$$
and thus
$$ 4^n m_2(E_n) \lesssim \sum_{Q_j \subset E_j: l(Q_j) = 4^{-j}} \beta_E(3Q_j)^2 l(Q_j) +
4^n m_2( (E_j \backslash E_{j+10}) \cap (K_n \times K_n ) ).$$
Summing in $j$ and using telescoping series, we obtain
$$ n 4^n m_2(E_n) \lesssim \sum_{Q} \beta_E(3Q)^2 l(Q) +
4^n m_2( K_n \times K_n ).$$
We can directly compute that $4^m m_2( K_m \times K_m ) = 1$.  Also, $E_m$ consists of $4^n m_2(E_n)$ squares of length $4^{-n}$, and the intersection of each such square with $E$ contributes a set of measure $O(4^{-n})$ to $A$. We thus have
$$ m(A) \lesssim 4^n m_2(E_n)$$
and the claim follows.
\end{proof}

The estimate in Proposition \ref{rector2} is significantly stronger than that in Proposition \ref{rector}, as the former becomes non-trivial as soon as $m-l \gg M$, whereas the latter is only non-trivial for $m-l \gg e^{C M^{100}}$.  
However, both estimates, when inserted into Theorem \ref{main}, give essentially the same result; replacing Proposition \ref{rector} by Proposition \ref{rector2} allows one to reduce the double-exponential in \eqref{mojo} to single-exponential, but after iteration this only affects the (unspecified) implied constant in the final bound for $\Fav(K_n \times K_n)$.

\end{document}